\documentclass[oneside,10pt]{amsart}
\usepackage{hyperref,amssymb,amsmath,amsthm,geometry,mdframed,color}
\usepackage[all]{xy}
\definecolor{theoremback}{rgb}{0.8235294, 0.8627451, 0.9137255}
\definecolor{defcolor}{rgb}{0.271, 0.443,0.663}
\hypersetup{
	colorlinks = true,
	linkcolor = defcolor,
	citecolor = defcolor
}
\usepackage{fancyhdr}
\newtheoremstyle{nthm}
{3pt}
{3pt}
{\itshape}
{0em}
{\bfseries}
{.}
{.5em}
{}

\newtheoremstyle{ndef}
{3pt}
{3pt}
{}
{0em}
{\bfseries}
{.}
{.5em}
{}

\newtheoremstyle{nrem}
{3pt}
{3pt}
{}
{0em}
{\itshape}
{.}
{.5em}
{}

\swapnumbers
\theoremstyle{nthm}

\newmdtheoremenv[linecolor=white,linewidth=2,backgroundcolor=theoremback]{thm}[subsection]{Theorem}

\newtheorem{prop}[subsection]{Proposition}
\newtheorem{cor}[subsection]{Corollary}

\theoremstyle{ndef}

\theoremstyle{nrem}

\newcommand{\Pcl}{\mathcal{P}}
\newcommand{\Z}{\mathbb{Z}}
\newcommand{\fgmod}[1]{#1\text{-}\mathbf{mod_{fg}}}
\DeclareMathOperator{\Jac}{Jac}

\DeclareMathOperator{\End}{End}

\begin{document}
\title[The $K$-Theory of Commuting Endomorphisms]{The $K$-Theory of Finitely Many Commuting Endomorphisms}
\author{Jason K.C. Polak}
\begin{abstract}
    For a field $k$ we compute the $K$-theory of the exact category of $k[t_1,\dots,t_n]$-modules that are finite-dimensional over $k$, generalising the work of Kelley and Spanier. 
\end{abstract}
\date{\today}
\maketitle

\section{Introduction}\label{sec:intro}

Let $A$ be a commutative ring and $\Pcl(A)$ the category of finitely-generated projective $A$-modules. Let $\End\Pcl(A)$ be the category whose objects are endomorphisms $P\to P$ where $P\in\Pcl(A)$, and whose morphisms $\alpha:(P\to P)\to (Q\to Q)$ are morphisms $\alpha:P\to Q$ in $\Pcl(A)$ such that the diagram
\begin{equation*}
    \xymatrix{
        P\ar[d]\ar[r]^\alpha & Q\ar[d]\\
        P\ar[r]^\alpha & Q
        }
\end{equation*}
commutes. The category $\End\Pcl(A)$ is naturally equivalent to the category of $A[t]$-modules that are finitely generated and projective as $A$-modules via the inclusion map $A\to A[t]$. The category $\End\Pcl(A)$ is an exact category and one would like to calculate its Quillen $K$-theory groups. This type of $K$-theory calculation falls under the more general problem of calculating for a ring homomorphism $R\to S$, the $K$-theory of the category of $S$-modules that are finitely generated and projective as $R$-modules.

The calculation of $K_i(\End\Pcl(A))$ was given by Kelley and Spanier \cite{KelleySpanier1968} when $A$ is a field and $i=0$. Almkvist \cite{Almkvist1978} did the calculation when $A$ is an arbitrary commutative ring and $i=0$, and when $A$ is a field and $i$ is arbitrary. To describe these computations, define the abelian group whose underlying set is
\begin{align*}
    \tilde{A}_0 = \left\{ \frac{1 + a_1t + \cdots + a_nt^n}{1 + b_1t + \cdots + b_mt^m} : a_i,b_j\in A\right \}
\end{align*}
and whose binary operation is given by the usual multiplication of rational functions. If $f:M\to M$ is an endormorphism of a finitely-generated projective $A$-module, then the characteristic polynomial $\lambda_t(f)$ may be defined by extending $f$ to an endomorphism of a free module, and then defining $\lambda_t(f) = \det(1 + tf)$; see \cite{Almkvist1978} for an alternative definition. We can use this map to describe the results of Kelley and Spanier and Almkvist:
\begin{thm}There is an isomorphism
    \begin{align*}
        K_0(\End\Pcl(A))&\longrightarrow K_0(A) \times \tilde{A}_0\\
    \end{align*}
    given on generators by
    \begin{align*}
        [M\xrightarrow{f} M] &\longmapsto ([M], \lambda_t(f)).
    \end{align*}
\end{thm}
In this paper, we generalise this result to the case of finitely many commuting endomorphisms and where $A=k$ is a field:
\begin{thm}
    The algebraic $K$-groups of the exact category $\End_n\Pcl(k)$ are given by
    \begin{align}
        K_i(\End_n\Pcl(k)) \cong \bigoplus_{M} K_i(k[t_1,\dots,t_n]/M)
    \end{align}
    where $M$ ranges over all the maximal ideals of the polynomial ring $k[t_1,\dots,t_n]$.
\end{thm}
We remark that our our proof is different than the proofs in \cite{Almkvist1978} and in \cite{KelleySpanier1968}. One feature is that our result for $i=0$ is not phrased in terms of a characteristic polynomial map, which is an advantage in the sense that it is more explicit in terms of the structure of the $K$-groups, but a disadvantage in the sense that working with products is more difficult. The result for $n=1$ and higher $K$-theory is \cite[Theorem 5.2]{Almkvist1978}, but again, our proof is entirely different.

\section{The Category $\End_n\Pcl(k)$}

Let $\End_n\Pcl(k)$ denote the exact category of $k[T]:=k[t_1,\dots,t_n]$-modules that are finitely generated as $k$-modules. In this section we calculate $K_i(\End_n\Pcl(k))$. A finite-dimensional $k$-vector space $V$ with any $n$ commuting endomorphisms $f_1,f_2,\dots,f_n$ of $V$ may also be considered as a $k[t_1,\dots,t_n]/I$ module where $I$ is the kernel of the map $k[t_1,\dots,t_n]\to k[f_1,\dots,f_n]$ given by $t_i\mapsto f_i$; hence:

\begin{prop}
    The category $\End_n\Pcl(k)$ is naturally equivalent to the filtered direct limit
    \begin{align}\label{eqn:dirlimit}
        \varinjlim_I \left(\fgmod{k[T]/I}\right)
    \end{align}
    of exact categories, where $I$ runs over the set of ideals of $k[T]$ such that the quotient $k[T]/I$ is finite-dimensional over $k$, and if $I\subseteq J$ are two such ideals, then the functor $\fgmod{k[T]/J}\to\fgmod{k[T]/I}$ is the forgetful functor induced by the quotient map $k[T]/I\to k[T]/J$.
\end{prop}
\begin{proof}
    The limit is indeed filtered, since $k[T]/(I\cap J)$ is finite-dimensional whenever $k[T]/I$ and $k[T]/J$ are finite-dimensional. Define a functor
    \begin{align*}
    F:\varinjlim_I \left(\fgmod{k[T]/I}\right)\to\End_n\Pcl(A)
\end{align*}
as follows. Any element in $\varinjlim_I \left(\fgmod{k[T]/I}\right)$ is represented by $V\in\fgmod{k[T]/I}$ for some $I$. We let $F(V)$ to be the $k[T]$ module given by the forgetful functor induced by the map $k[T]\to k[T]/I$. This functor is well-defined because $k[T]/I$ is finite-dimensional, and it is easy to see that $F$ gives the required natural equivalence.
\end{proof}
Using this observation and the result that taking $K$-groups of exact categories commutes with filtered direct limits \cite[\S2]{Quillen1972}, we obtain the following corollary.
\begin{cor}
    The $K$-theory of the category $\End_n\Pcl(A)$ may be calculated as the direct limit
    \begin{align}
        K_i(\End_n\Pcl(A)) \cong \varinjlim_I K_i\left(\fgmod{k[T]/I}\right)
    \end{align}
\end{cor}
For any ring $R$, let $\Jac(R)$ denote the Jacobson radical of $R$. Any surjective ring homomorphism $R\to S$ induces a surjective ring homomorphism $R/\Jac(R)\to S/\Jac(S)$, and a commutative diagram of rings
\begin{equation*}
    \xymatrix{
        R\ar[r]\ar[d] & S\ar[d]\\
        R/\Jac(R)\ar[r] & S/\Jac(S)
    }
\end{equation*}
In turn, whenever $S$ is finitely generated as an $R$-module, we have a commutative diagram of forgetful functors
\begin{equation}\label{eqn:forgetful}
    \xymatrix{
        \fgmod{R} & \fgmod{S}\ar[l]\\
        \fgmod{R/\Jac(R)}\ar[u] & \fgmod{S/\Jac(S)}\ar[l]\ar[u]
        }
\end{equation}
In particular, this applies to the ring homomorphisms $k[T]/I\to k[T]/J$ for $I\subseteq J$ appearing in the direct limit of~\eqref{eqn:dirlimit}. 
\begin{prop}\label{thm:lastreduction}
    The natural transformation of the direct limits of categories induced by the forgetful functors as in \eqref{eqn:forgetful} for $R = k[T]/I$ induces an isomorphism algebraic $K$-groups
\begin{align*}
    \varinjlim_I K_i\left( \fgmod{\frac{k[T]/I}{\Jac(k[T]/I)}}\right)\overset{\sim}{\longrightarrow} \varinjlim_I K_i\left(\fgmod{k[T]/I}\right).
\end{align*}
\end{prop}
\begin{proof}
    For any Artinian ring $R$, the Jacobson radical $\Jac(R)$ is nilpotent (e.g. \cite[Theorem 4.12]{Lam1991}), and so devissage \cite[\S5, Theorem 4]{Quillen1972} shows that the inclusion $\fgmod{R/\Jac(R)}\to \fgmod{R}$ induces an isomorphism 
    \begin{align*}
    K_i(\fgmod{R/\Jac(R)})\to K_i(\fgmod{R})
\end{align*}
    (see \cite[Page 439]{Weibel2013} for more details). In particular, this applies to $R = k[T]/I$ where $k[T]/I$ is finite-dimensional over $k$.
\end{proof}
\begin{thm}\label{thm:maintheorem}
    The algebraic $K$-groups of the exact category $\End_n\Pcl(k)$ are given by
    \begin{align}
        K_i(\End_n\Pcl(k)) \cong \bigoplus_{M} K_i(k[t_1,\dots,t_n]/M)
    \end{align}
    where $M$ ranges over all the maximal ideals of the polynomial ring $k[t_1,\dots,t_n]$.
\end{thm}
\begin{proof}
    We must calculate the limit
    \begin{align}\label{eqn:finaldirectlimit}
\varinjlim_I K_i\left( \fgmod{\frac{k[T]/I}{\Jac(k[T]/I)}}\right)
\end{align}
given in Proposition~\ref{thm:lastreduction}. So, fix an ideal $I$ such that $k[T]/I$ is finite-dimensional. Then there are finitely many maximal ideals $M_1,\dots,M_k$ of $k[T]$ that contain $I$ and 
 \begin{align*}
     k[T]/I \cong \bigoplus_{j=1}^k (k[T]/I)_{M_j}
 \end{align*}
 via the obvious map (e.g. \cite[Theorem 5.20]{Goertz2010}). Here, by $(k[T]/I)_{M_i}$, we abuse notation and mean the localization of $k[T]/I$ away from the maximal ideal $M_i(k[T]/I)$. If $J$ is an ideal containing $I$, then there is a subset $\{N_1,\dots,N_\ell\}$ of the maximal ideals $\{M_1,\dots,M_k\}$ that contain $J$ and the quotient homomorphism $k[T]/I\to k[T]/J$ induces the map
\begin{align*}
    \bigoplus_{j=1}^k (k[T]/I)_{M_j}\longrightarrow\bigoplus_{j=1}^\ell (k[T]/J)_{N_j}
\end{align*}
which must be the projection map. The induced map on $K$-groups that fits into the direct limit in~\eqref{eqn:finaldirectlimit}, by  is then the map
\begin{align*}
    \bigoplus_{j=1}^\ell K_i(k[T]/N_j)\longrightarrow\bigoplus_{j=1}^k K_i(k[T]/M_j).
\end{align*}
given by the sum of the inclusion maps. Taking the direct limit gives the stated result, once we note that $k[T]/M$ is finite-dimensional over $k$ for any maximal ideal $M$.
\end{proof}
In particular, we obtain the following amusing, possibly well-known result.
\begin{cor}
    For a field $k$, there is an isomorphism of abelian groups
\begin{align*}
    \tilde{k}_0 := \left\{ \frac{1 + a_1t + \cdots + a_nt^n}{1 + b_1t + \cdots + b_mt^m} : a_i,b_j\in k\right \} \cong \bigoplus_{|\mathrm{Spec}(k[T])| - 1}\Z.
\end{align*}
\end{cor}

\bibliographystyle{alpha}
\bibliography{/media/jpolak/KINGSTON/math_papers/polakmain.bib}

\begin{thebibliography}{Lam91}

\bibitem[Alm78]{Almkvist1978}
Gert Almkvist.
\newblock K-theory of endomorphisms.
\newblock {\em Journal of Algebra}, 55(2):308--340, 1978.

\bibitem[GW10]{Goertz2010}
Ulrich G{\"o}rtz and Torsten Wedhorn.
\newblock {\em Algebraic geometry {I}: Schemes with examples and exercises}.
\newblock Advanced Lectures in Mathematics. Vieweg + Teubner, Wiesbaden, 2010.

\bibitem[KS68]{KelleySpanier1968}
John~L. Kelley and Edwin~H. Spanier.
\newblock Euler characteristics.
\newblock {\em Pacific Journal of Mathematics}, 26(2):317--339, 1968.

\bibitem[Lam91]{Lam1991}
T.Y. Lam.
\newblock {\em A First Course in Noncommutative Rings}.
\newblock Graduate Texts in Mathematics. Springer Verlag, 1991.

\bibitem[Qui72]{Quillen1972}
Daniel Quillen.
\newblock Higher algebraic k-theory: I.
\newblock In {\em Algebraic K-Theory I: Higher K-Theories}, Springer Lecture
  Notes in Mathematics, pages 85--147. Springer, Berlin, 1972.

\bibitem[Wei13]{Weibel2013}
Charles~A. Weibel.
\newblock {\em The $K$-Book: An Introduction to Algebraic $K$-Theory}.
\newblock American Mathematical Society, 2013.

\end{thebibliography}

\end{document}